\documentclass[12pt,amssymb, amscd, amstheorem,, amsmath, amstext]{amsart}

\usepackage{amsmath}
\usepackage{amssymb}
\usepackage{amsthm}
\usepackage{color}
\usepackage[all]{xy}
\usepackage{graphicx}

\newtheorem{theorem}{Theorem}
\newtheorem{definition}{Definition}
\newtheorem{lemma}{Lemma}
\newtheorem{corollary}{Corollary}
\newtheorem{proposition}{Proposition}
\newtheorem{remark}{Remark}

\newcommand{\C}{\mathbb C }

\def\T{\mathcal T}
\def\T2{\mathcal{AV}2}
\def\TE{\mathcal TE}

\def\M{\mathcal M}
\def\P{\mathcal P}
\def\E{\mathcal E}

\def\CC{\mathbb C}

\begin{document}

\title[Bounded Geometry and Characterization]{Bounded Geometry and Characterization
of Some Transcendental Entire and Meromorphic Maps}

\author{Tao Chen, Yunping Jiang, and Linda Keen}

\begin{abstract}
We define two classes of topological infinite degree   covering maps modeled on two families of transcendental holomorphic maps.  The first, which we call exponential maps of type $(p,q)$, are branched covers and is modeled on transcendental entire maps of the form 
$P e^{Q}$, where $P$ and $Q$ are polynomials of degrees $p$ and $q$. The second is the class of universal covering maps from the plane to the sphere with two removed points modeled on transcendental meromorphic maps with two asymptotic values.  The problem we address is to give a combinatorial characterization of  the holomorphic maps contained in these classes whose post-singular sets are finite. 
  The main results in this paper are that a post-singularly finite topological exponential map of type $(0,1)$ or a certain post-singularly finite topological exponential map of type $(p,1)$ or a post-singularly finite universal covering map  from the plane to the sphere with two points removed is combinatorially equivalent to a holomorphic same type map if and only if this map has {\em bounded geometry}.
\end{abstract}
\maketitle

\section{Introduction}
\label{sec:intro}

Thurston proved  that a post-critically finite degree $d\geq 2$  branched covering of
the sphere, with hyperbolic orbifold, is either combinatorially equivalent
to a rational map or there is a  topological obstruction, now called a ``Thurston obstruction'' (see~\cite{T,DH,Ji} for definition). 
The proof uses an iteration scheme defined for an appropriate Teichm\"uller space. As presented in~\cite{DH}, the proof
is divided into two steps.  For any initial point in the Teichm\"uller space, the iteration scheme gives a sequence of points 
in the Teichm\"uller space. This sequence reduces to a corresponding sequence in the corresponding moduli space. 
The first step is to prove that the map is combinatorially equivalent to a rational map is equivalent to that the corresponding sequence is contained in a compact subset in the moduli space. The second step is to prove that 
the corresponding sequence is contained in a compact subset in the moduli space is equivalent to that there is no ``Thurston obstruction''.
Both the proofs depend essentially on the finiteness of degree and the finiteness of the dimension of the Teichm\"uller space. 
The reason that the proof of the first step depends on the finiteness of degree essentially is the key lemma~\cite[Lemma 5.2]{DH} in the proof, while the proof of this key lemma depends on the finiteness of degree essentially.  This key lemma guarantees that the sequence in the Techm\"uller space is contained in a compact subset in the Teichm\"uller space provided the corresponding sequence in the moduli space is contained in a compact subset in the moduli space. Thus although Thurston's iteration scheme is well defined for topological transcendental maps, the proof of the first step as well as the proof of the second step as presented in~\cite{DH} can not applied to this case due to the infiniteness of degree.  
Another framework to prove the first step is outlined in~\cite{Ji} for branched coverings of finite degree. This framework avoids the key lemma~\cite[Lemma 5.2]{DH} in the proof.  Therefore, it does not depend on the finiteness of degree.  The main purpose of this paper is to show that this framework works for topological transcendental maps too. 

More precisely, in this paper we define two topological classes of covering maps of the plane.  The first,  called
{\em topological exponential maps of type $(p,q)$} and denoted by $\TE_{p,q}$, where $p\geq 0$ and $q\geq 1$ are  branched covers.  
We follow the framework given in~\cite{Ji} to study the  problem of combinatorially characterizing an entire map 
$P e^{Q}$, where $P$ are $Q$ are polynomials of degrees $p$ and $q$, using an {\em iteration scheme defined by Thurston} and
a {\em bounded geometry condition}.  The second, which we treat in a similar way,  is  the topological class of universal covering maps from the plane to the sphere with two removed points;  we call the elements  topological transcendental maps with two asymptotic values and we denote the space by $\T2$. Note that $\TE_{0,1}\subset \T2$.
We first show that  an element $f \in {\TE}_{p,q}\cup \T2$ with finite post-singular set
is combinatorially equivalent to a holomorphic same type map if and only if
it has bounded geometry and satisfies a {\em compactness condition}. 
Thus to complete the characterization, we only need to check that  bounded 
geometry actually implies  compactness --- we show this for some $f\in \TE_{p,q}$ and for $f\in \T2$. More precisely, we first prove    

\medskip
\begin{theorem}~\label{main1}
A post-singularly finite map $f$ in $\TE_{p,q}\cup \T2$  is combinatorially equivalent to a post-singularly finite entire map of the form $E=Pe^{Q}$ or a post-singularly finite meromorphic map with two asymptotic values  if and only if it has bounded geometry and satisfies the compactness condition. The realization is unique up to conjugation by an affine map of the plane.
\end{theorem}

We then prove our main theorems,  that bounded geometry implies the  compactness hypothesis holds in certain cases.  The first one is 

\medskip
\begin{theorem}~\label{main2}
A post-singularly finite map $f$ in $\TE_{p,1}$, $p\geq 1$, with only one non-zero simple branch point $c$ such that either $c$ is periodic or $c$ and $f(c)$ are both not periodic, is combinatorially equivalent to a unique post-singularly finite entire map of the form $ \alpha z^{p}e^{\lambda z}$, where $\alpha=(-\lambda/p)^{p}e^{- \lambda (-p/\lambda)^{p}}$, respectively,  if and only if it has  bounded geometry.  
\end{theorem}

This result is the first one to use the Thurston iteration scheme to characterize a transcendental entire map with critical points and is thus completely new.  Therefore, we  give a detailed proof of Theorem~\ref{main2}. 

The second is  the content of our paper~\cite{CJK}.

\medskip
\begin{theorem}~\label{main3}
A post-singularly finite map $f$ in $\T2$ is combinatorially equivalent
to a post-singularly finite transcendental meromorphic function $g$ with constant Schwarzian derivative if and only if it has bounded geometry.
The realization is unique up to conjugation by an affine map of the plane.
\end{theorem}

Because the proof of this theorem  is similar to the proof of Theorem~\ref{main2}, we include an outline of it here.  
The reader who is interested in the full  proof  is referred  to~\cite{CJK}. 

\medskip
Our techniques involve adapting the Thurston iteration scheme to our situation.  We work with a fixed normalization.   The proof of Theorem~\ref{main1}  applies to an arbitrary post-singularly finite map  in either $ \TE_{p,q}$ or  $ \T2$.  It shows that bounded geometry together with the  assumption of compactness implies the convergence of the iteration scheme to an entire or meromorphic map of the same type (see section \ref{sec:suff}). Its proof involves  an analysis of  quadratic differentials associated to the functions in  the iteration scheme.   To prove Theorems~\ref{main2} and~\ref{main3}, which apply to the special cases, we need a topological constraint (see section~\ref{sec:proofmt}).   We then prove  that   bounded geometry together with the topological constraint implies compactness. 
 
 \medskip
Finally, we mention a remark about the second step in Thurston's iteration scheme. 
It is still a problem to prove that  the corresponding sequence is contained in a compact subset 
in the moduli space is equivalent to that there is no ``Thurston obstruction'' for 
topological exponential maps of type $(p,q)$ and for topological transcendental maps 
with two asymptotic values. In a recent paper~\cite{HSS}, Hubbard, Schleicher, and Shishikura used 
a Levy cycle, a very special type ``Thurston obstruction'', to investigate whether
the corresponding sequence is contained in a compact subset in the moduli space for 
a topological exponential map of type $(0,1)$. The full answer for the second step by using ``Thurston obstruction'' 
is still open, even for this special case.

\medskip
The paper is organized as follows. In \S2, we review the covering properties of $(p,q)$-exponential maps $E=Pe^{Q}$.
In \S3, we define the family $\TE_{p,q}$ of $(p,q)$-topological exponential maps $f$.
In \S4, we review the properties of meromorphic maps with two asymptotic values. In \S5, we describe the space of universal covering maps from the Eucildean plane to the sphere with two  removed points. In \S6,  we define  combinatorial equivalence between post-singularly
finite maps in ${\mathcal TE}_{p,q}$ or  $\T2$ and prove  there is a local quasiconformal map in every combinatorial equivalence class. 
In \S7, we define the Teichm\"uller space $T_{f}$ for a post-singularly
finite $(p,q)$-topological exponential map $f$ or a post-singularly finite map in $\T2$. 
In \S8, we introduce the induced map $\sigma_{f}$ from  the Teichm\"uller space $T_{f}$ into itself;  this is the crux of the Thurston iteration scheme.
In \S9, we define the concept of  ``bounded geometry''  and in \S10 we prove the necessity of the bounded geometry condition.
 In \S11, we give the proof of sufficiency assuming compactness for any post-singularly finite map $f$ in either ${\mathcal TE}_{p,q}$ or $\T2$.
 In \S12, we define a topological constraint for the maps in Theorem~\ref{main2} and Theorem~\ref{main3}; this involves defining markings and the winding number of a  homotopy class.   We prove that the winding numbers and the homotopy classes are unchanged under iteration of the map $\sigma_f.$ Furthermore, in \S13
 we prove that  bounded geometry together with the topological constraint implies  compactness. This completes the proofs of Theorem~\ref{main2} and Theorem~\ref{main3}.

\medskip
 \medskip
{\bf Acknowledgement:} This work is partially supported by the PSC-CUNY award program, the CUNY collaborative incentive research grant program, 
the Simons Foundation collaboration award program, the collaboration grant (\#11171121) from the NSF of China, and the Academy of Mathematics and Systems Science and the Morningside
Center of Mathematics at the Chinese Academy of Sciences.

\section{The space $\E_{p,q}$ of $(p,q)$-Exponential Maps}
\label{sec:epq}

\label{pq-exp maps}
We use the following notation:
  ${\mathbb C}$ is the complex plane, $\hat{\mathbb C}$ is the Riemann sphere and
  ${\mathbb C}^{*}$ is the complex plane  punctured at the origin.

A {\em $(p,q)$-exponential map} is an entire function of the form  $E =Pe^{Q}$ where $P$ and $Q$ are  polynomials of degrees $p\geq 0$ and $q\geq 0$ respectively such that  $p+q\geq 1$.  We use the notation ${\mathcal E}_{p,q}$ for the set of $(p,q)$-exponential maps.

Note that if $P(z)=a_0 + a_1 z + \ldots a_pz^p$, $Q(z)=b_0 + b_1 z + \ldots b_q z^q$, $\widehat{P}(z)=e^{b_0} P(z)$ and $\widehat{Q}(z)=Q(z)-b_0$ then
$$P(z) e^{Q(z)} = \widehat{P}e^{\widehat{Q}(z)}.$$
To avoid this ambiguity we always assume $b_0=0$.
If $q=0$, then $E$ is a polynomial of degree $p$.    Otherwise, $E$ is a transcendental entire function with essential singularity at infinity.

The growth rate of an entire function $f$ is defined as
$$
\limsup_{r\to \infty} \frac{\log \log M(r)}{\log r}
$$
where $M(r) =\sup_{|z|=r} |f(z)|$. It is easy to see that the growth rate of $E$ is $q$.

Recall the following definitions: 
\begin{definition}
 Given an entire or meromorphic function $g$, the point $v$ is an {\em asymptotic value} of $g$ if there is a path $\gamma(t)$ such that
 $\lim_{t \rightarrow 1}\gamma(t)=\infty$ and $\lim_{t\rightarrow 1}g(\gamma(t))=v$.  It 
  is a {\em logarithmic singularity for the map $g^{-1}$}
 if there is a neighborhood $U_v$ and a component $V$ of $g^{-1}(U_v \setminus \{v\}) $ such that the map
 $g: V \rightarrow U_v \setminus \{v\}$ is a holomorphic universal covering map. 
    If an asymptotic value is isolated, it is a logarithmic singularity.  (This will always be the case in this paper.) The domain $V$ is called an {\em asymptotic tract for $v$}.
 A point may be an asymptotic value for more than one asymptotic tract.
 An asymptotic value may be an omitted value.
 \end{definition}

\begin{definition}  Given a holomorphic or meromorphic function  $g$, the point $v$ is an {\em algebraic singularity
for the map $g^{-1}$} if there is a neighborhood $U_v$ such that for every component $V_i$
of $g^{-1}(U_v ) $ the map $g: V_i \rightarrow U_v $ is a degree $d_{V_i}$ branched covering map
and $d_{V_i}>1$ for finitely many components $V_{1}, \ldots V_{n}$. For these components,
if $c_i \in V_i$ satisfies $g(c_{i})=v$ then $g'(c_i)=0$;
that is $c_i$ is a {\em critical point of $g$} for $i=1, \ldots, n$ and $v$ is a {\em critical value}.
\end{definition}

An entire or meromorphic function is not a homeomorphism,
it must have at least one, respectively, two singular points  (i.e., critical points and asymptotic values) and, by the big Picard theorem,
no entire function can omit more than one value and no transcendental meromorphic function $g:\C \rightarrow \hat\C$ can omit more than two values.  

For functions $E$ in ${\mathcal E}_{p,q}$ we have

\begin{proposition}
If $q\geq 1$, $E$ has $2q$ distinct asymptotic tracts that are separated by $2q$  rays.  Each tract maps to a punctured neighborhood of either zero or infinity and these are the only asymptotic values.
\end{proposition}

\begin{proof}  From the growth rate of $E$ we see that for $|z|$ large, the behavior of the exponential dominates.   Since $Q(z) = b_q z^q + \mbox{ lower order terms} $,  in a neighborhood of infinity there are $2q$ branches of $\Re Q = 0$ asymptotic to equally spaced rays.  In the $2q$ sectors defined by these rays the signs of $\Re Q$ alternate.  If $\gamma(t)$ is a curve
such that $\lim_{t \to \infty} \gamma(t) = \infty$ and $\gamma(t)$ stays in one sector for all large $t$, then either $\lim_{t \to \infty}E(\gamma(t))=0$ or $\lim_{t \to \infty}E(\gamma(t))=\infty$, as $\Re Q$ is negative or positive in the sector.  It follows that there are exactly $q$ sectors that are asymptotic tracts for $0$ and $q$ sectors that are asymptotic tracts for infinity.    Because the complement of these tracts in a punctured neighborhood of infinity consists entirely of these rays, there can be no other asymptotic tracts.  \end{proof}

\begin{remark} The directions dividing the asymptotic tracts  are called {\em Julia rays} or {\em Julia directions} for $E$.  If $\gamma(t)$ tends to infinity along a Julia ray,  $E(\gamma(t))$ remains in a compact domain in the plane.   It spirals infinitely often around the origin.  \end{remark}

Two
$(p,q)$-exponential maps $E_{1}$ and $E_{2}$ are conformally equivalent if they are conjugate by a conformal automorphism $M$ of the Riemann sphere $\hat{\mathbb C}$, that is, $E_{1} =M\circ E_{2}\circ M^{-1}$. The automorphism $M$ must be a M\"obius transformation and it must fix both $0$ and $\infty$ so that it must be the affine stretch  map $M(z)=az$, $a \neq 0$. We are interested in conformal equivalence classes of maps, so by abuse of notation, we treat conformally equivalent $(p,q)$-exponential maps $E_{1}$ and $E_{2}$ as the same.

The critical points of $E=Pe^{Q}$ are the roots of $P'+PQ'=0$. Therefore, $E$ has $p+q-1$ critical points counted with multiplicity which we denote by
$$
\Omega_{E}=\{ c_{1}, \cdots, c_{p+q-1}\}.
$$
Note that if $E(z)=0$ then $P(z)=0$.  This in turn implies that if $c\in \Omega_E$ maps to $0$, then $c$ must also be a critical point of $P$.
Since $P$ has only $p-1$ critical points counted with multiplicity,
there must be at least $q$ points (counted with multiplicity) in $\Omega$ which are not mapped to $0$.
Denote by
$$
\Omega_{E,0}=\{ c_{1}, \cdots, c_{k}\}, \quad k\leq p-1,
$$
 the (possibly empty) subset of  $\Omega_E$ consisting of critical points such that $E(c_i)=0$.  Denote its complement in $\Omega_E$ by
$$
\Omega_{E,1} =\Omega_{E}\setminus \Omega_{E,0}=\{ c_{k+1}, \cdots, c_{p+q-1}\}.
$$

When $q=0$, $E$ is a polynomial. The {\em post-singular set} in this special case is  the same as the {\em post-critical set}. It is defined as
$$
P_{E} =\overline{\cup_{n\geq 1} E^{n}(\Omega_{E})}\cup\{\infty\}.
$$
To avoid trivial cases here we will assume that $\#(P_{E}) \geq 4$.  Conjugating by an affine map $z \to az+b$ of the complex plane,
	we normalize so that  $0, 1\in P_{E}$.

When $q=1$ and $p=0$, $\Omega_{E}=\emptyset$ and $\E_{0,1}$ consists of exponential maps $\alpha e^{\lambda z}$, $\alpha, \lambda \in \CC^*$. 
The {\em post-singular set} in this special case is defined as
$$
P_{E} =\overline{\cup_{n\geq 0} E^{n}(0)} \cup \{\infty\}.
$$
Conjugating by an affine stretch $z \mapsto \alpha z$ of the complex plane,
we normalize so that $E(0)=1$. Note that after this normalization the family takes the form
$e^{\lambda  z}$ , $\lambda \in \CC^{*}$.

When $q\geq 2$ and $p=0$ or when  $q\geq 1$ and $p\geq 1$, $\Omega_{E,1}$ is a non-empty set. Let
$$
{\mathcal V} =E(\Omega_{E,1}) =\{ v_{1}, \cdots, v_{m}\}
$$
denote the set of non-zero critical values of $E$.
The  {\em post-singular set}  for $E$  in  the general case  is now defined as

$$
P_{E} = \overline{\cup_{n\geq 0} E^{n}({\mathcal V}\cup \{ 0 \})} \cup \{\infty\}.
$$
We normalize  as follows:  \\

 If $E$ does not fix $0$, which is always true  if $q\geq 2$ and $p=0$,
we conjugate by an affine stretch $z \rightarrow az$ so that  $E(0)=1$.

If  $E(0)=0$,  there is a  critical point in  $c_{k+1}$ in $\Omega_{E,1}$  with $c_{k+1}\not=0$ and $v_{1}=E(c_{k+1})\not= 0$.
In this case we normalize so that $v_{1}=1$.  The family $\E_{1,1}$ consists of functions of the form  $\alpha z e^{\lambda z}$.
 After normalization  they  take  the form
$$
-\lambda e ze^{\lambda z}.
$$
An important family we consider in this paper is the family in $\E_{p,1}$, $p\geq 1$, where each map in this family 
has only one non-zero simple critical point. After normalization, the functions in this family take the form
$$
E(z)= \alpha z^{p} e^{\lambda z}, \quad \alpha=\Big( -\frac{\lambda}{p}\Big)^{p}  e^{p}.
$$
This is the  family for which we have the strongest new results. 

\section{Topological Exponential Maps of Type $(p,q)$}
\label{sec:Tpq}
We use the notation  ${\mathbb R}^{2}$ for the Euclidean plane.   We define the space ${\TE}_{p,q}$ of {\em topological exponential maps of type $(p,q)$}  with $p+q\geq 1$.    These are branched coverings with a single finite asymptotic value, normalized to be at zero, modeled on the maps in the holomorphic family $\E_{p,q}$.  In~\cite{Z} Zakeri gives a description  of the covering properties for the family $\E_{p,q}$.  Our definition of the covering properties for the family ${\TE}_{p,q}$ is modeled on $\E_{p,q}$.  

If $q=0$, then ${\mathcal TE}_{p,0}$ consists of all topological polynomials $P$ of degree $p$:  these are degree $p$ branched coverings of the sphere such that $f^{-1}(\infty)=\{\infty\}$.

If $q=1$ and $p=0$, the space ${\TE}_{0,1}$ consists of universal covering maps $f: {\mathbb R}^{2}\to {\mathbb R}^{2}\setminus \{0\}$.  These are discussed at length in \cite{HSS}, where they are called topological exponential maps.

The polynomials $P$ and $Q$ contribute differently to the covering properties of maps in $\E_{p,q}$.  As we saw,  the degree of $Q$ controls the growth and behavior at infinity.    To see this we first define  the space ${\TE}_{0,q}$ 
using maps  $e^{Q}$ as our model.

 \begin{definition}
If $q\geq 2$ and $p=0$, the space ${\TE}_{0,q}$ consists of topological branched covering maps $f: {\mathbb R}^{2}\to {\mathbb R}^{2}\setminus \{0\}$ satisfying the following conditions:
\begin{itemize}
\item[i)] The set of branch points,  $\Omega_{f} =\{c\in {\mathbb R}^{2}\;|\; \deg_{c}f\geq 2\}$  consists of $q-1$ points counted with multiplicity.
\item[ii)] Let ${\mathcal V} =\{ v_{1}, \cdots, v_{m}\} =f(\Omega_{f})\subset {\mathbb R}^{2}\setminus \{0\}$ be the set
           of distinct images of the branch points. For $i=1, \ldots, m$, let $L_{i}$  be a smooth topological ray in ${\mathbb R}^{2}\setminus \{0\}$ starting at  $v_{i}$  and extending to $\infty$ such that the collection of rays
            $\{L_{1}, \cdots, L_{m}\}$ are pairwise disjoint. Then
        \begin{enumerate}
              \item $f^{-1}(L_{i})$ consists of infinitely many  rays starting at points in the preimage set $f^{-1}(v_{i})$. If $x\in f^{-1}(v_{i})\cap \Omega_{f}$, there are $d_{x}=\deg_{x} f$ rays meeting at $x$ called {\em critical rays}.     If $x\in f^{-1}(v_{i})\setminus \Omega_{f}$, there is only one ray emanating from  $x$; it is called a {\em non-critical ray}.
              Set
              $$
              W=\mathbb R^2 \setminus  ( \cup_{i=1}^m L_i \cup \{0\}).
              $$
              \item The set of critical rays meeting at points in $\Omega_{f}$ divides $f^{-1}(W)$ into $q=1+\sum_{c\in \Omega_{f}} (d_{c}-1)$ open unbounded   connected components $W_{1}, \cdots, W_{q}$.
              \item[(3)] $f: W_{i}\to W $ is a universal covering for each $1\leq i\leq q$.
        \end{enumerate}
\end{itemize} \end{definition}

Note that the map restricted to each $W_i$ is a topological  model for the exponential map $z\mapsto e^{z}$ and the local degree at the critical points determines the number of $W_i$ attached at the point.

 We now define the space    ${\TE}_{p,q}$ in full generality where we assume $p>0$ and there is additional  behavior modeled on the role of  the new  critical points of $P e^{Q}$ introduced by  the non-constant polynomial $P$.

\medskip
\begin{definition}~\label{topexpdef}
If  $q\geq 1$ and $p\geq 1$, the space ${\TE}_{p,q}$ consists of topological branched covering maps $f: {\mathbb R}^{2}\to {\mathbb R}^{2}$ satisfying the following conditions:
\begin{itemize}
\item[i)] $f^{-1}(0)$ consists of $p$ points counted with multiplicity.
\item[ii)] The set of branch points, $\Omega_{f} =\{c\in {\mathbb R}^{2}\;|\; \deg_{c}f\geq 2\}$ consists of $p+q-1$  points counted with multiplicity.
\item[iii)] Let $\Omega_{f, 0} = \Omega_{f} \cap f^{-1}(0)$  be the $k<p$ branch points that map to $0$ and $\Omega_{f,1} =\Omega_{f}\setminus \Omega_{f,0}$  the $p+q-1-k$ branch points  that do not.   Note that $\Omega_{f,1}$ contains at least $q$ points and
  ${\mathcal V} =\{ v_{1}, \cdots, v_{m}\} =f(\Omega_{f,1})$  is contained in ${\mathbb R}^{2}\setminus \{0\}$. For $i=1, \ldots, m$, let $L_{i}$  be a smooth topological ray in ${\mathbb R}^{2}\setminus \{0\}$ starting at $v_{i}$ and extending to $\infty$ such that  the collection of rays  $\{L_{1}, \cdots, L_{m}\}$ are pairwise disjoint. Then
          \begin{enumerate}
                \item $f^{-1}(L_{i})$ consists of infinitely many  rays starting at points in the pre-image set $f^{-1}(v_{i})$.   If $x\in f^{-1}(v_{i}) \setminus \Omega_{f,1}$, there is only one ray emanating from $x$; this is a {\em non-critical ray}.  If $x\in f^{-1}(v_{i})\cap \Omega_{f, 1}$, there are $d_{x}=\deg_{x} f$  {\em critical rays} meeting at $x$.   Set $$W=\mathbb R^2 \setminus  ( \cup_{i=1}^m(L_i) \cup \{0\}).$$
                \item The collection of all critical rays meeting at points in $\Omega_{f,1}$ divides $f^{-1}(W)$ into $l=p+q-k=1+\sum_{c\in \Omega_{f,1}} (d_{c}-1)$  open unbounded connected components.
                \item Set $f^{-1}(0) =\{ a_{i}\}_{i=1}^{p-k}$  where the $a_i$ are distinct.   Each $a_i$ is contained in a distinct component of $f^{-1}(W)$;  label these components $W_{i,0}$, $i=1, \ldots p-k$.  Then the restriction   $f: W_{i,0} \setminus \{a_i\} \rightarrow W$ is an unbranched covering map  of degree $d_i=deg_{a_i}f$ where $d_i>1$ if $a_i \in \Omega_{f,0}$ and $d_i=1$ otherwise.
                  \item Label the remaining $q$ connected
                  components of $f^{-1}(W)$ by $W_{j,1}$, $j=1, \ldots, q$.
                  Then the restriction $f: {W_{j,1}} \rightarrow W$
                  is a universal covering map.
          \end{enumerate}
\end{itemize}
\end{definition}
 
From~\cite[Section 3]{Z}, we know that the $(p,q)$-exponential maps are topological exponential maps  of type $(p,q)$ as we defined. The converse is also true.

\vspace*{5pt}
\begin{theorem}~\label{topexp}
Suppose $f\in {\TE}_{p,q}$ is analytic. Then $f=Pe^{Q}$ for two polynomials $P$ and $Q$ of degrees $p$ and $q$. That is, an analytic topological exponential map of type $(p,q)$ is a $(p,q)$-exponential map.
\end{theorem}

\begin{proof}
If $q=0$, then $f$ is a polynomial $P$ of degree $p$.

If $q\geq 1$, then $f$ is an entire function with $p$ roots, counted with multiplicity.
Every  such function can be expressed as
$$
f (z)= P(z) e^{g(z)}
$$
where $P$ is a polynomial of degree $p$ and $g$ is some entire function (see~\cite[Section 2.3]{Al}).

Consider  $$f'(z) = (P(z)g'(z)+P'(z))e^{g(z)}.$$
It is also an entire function,  and by assumption it has $p+q-1$ roots so that $Pg'+P'$ is a polynomial of degree $p+q-1$.  It follows that $g'$ is a polynomial of degree $q-1$ and $g=Q$ is a polynomial of degree $q$.
\end{proof}

Note that if $f \in  {\TE}_{p,q}$, $ q \neq 0$,  the origin plays a special role:  it is the only point with no or finitely many pre-images.  The conjugate of $f$ by $z \mapsto az$,  $ a \in {\mathbb C}^*$, is also in  ${\mathcal TE}_{p,q}$;  conjugate maps are conformally equivalent.

For $f\in {\TE}_{p,q}$, we define the {\em post-singular set} as follows:
\begin{itemize}
\item[i)] When $q=0$,  $E$ is a polynomial and, as mentioned in the introduction,  is treated elsewhere.   We therefore always assume $q \geq 1$.
\item[ii)] When $q=1$ and $p=0$, the {\em post-singular set} is
$$
P_{f} =\overline{\cup_{n\geq 0} f^{n}(0)}\cup\{\infty\}.
$$
We normalize so that $f(0)=1\in P_{f}$.
\item[iii)] When $q\geq 1$ and $p\geq 1$, the set of branch points is
$$\Omega_{f} =\{ c\in {\mathbb R}^{2}\;|\; \deg_{c} f \geq 2\}$$ and the {\em post-singular set} is
$$
P_{f} = \overline{\cup_{n\geq 0} f^{n}({\mathcal V}\cup \{ 0 \})} \cup \{\infty\}.
$$
If $q>1$ or if $q=1$ and $f(0) \neq 0$, we normalize so that $f(0)=1\in P_{f}$.
If $f(0)=0$, then, by the assumption $q \geq 1$, there is  a branch point $c_{k+1}\not=0$ such that $v_{1}=f(c_{k+1})\not= 0$.
We normalize so that $v_{1}=1$.
\end{itemize}
To avoid trivial cases we assume that $\#(P_{f}) \geq 4$.

It is clear that, in any case, $P_{f}$ is forward invariant, that is,
$$
f(P_{f}\setminus \{\infty\})\cup \{\infty\} \subseteq P_{f}
$$
or equivalently,
$$
f^{-1} (P_{f} \setminus \{\infty\})\cup \{\infty\} \supset P_{f}.
$$
Note that since we assume $q\geq 1$, $f^{-1}(P_{f}\setminus \{\infty\}) \setminus (P_{f}\setminus \{\infty\})$
contains infinitely many points.

\begin{definition}
We call $f\in {\TE}_{p,q}$  {\em post-singularly finite} if $\#(P_{f})<\infty$.
\end{definition}

\section{The Space $\M_2$} \label{sec:merom}
In this section we define the space of meromorphic functions $\M_2$.
It is the model for the  more general space of topological functions $\T2$ that we define in the next section.

The space $\M_2$ consists of meromorphic functions whose only singular values are its omitted values.
More precisely,

\begin{definition}
The space $\M_2$ consists of meromorphic functions $g:\C \rightarrow \hat\C$
with exactly two asymptotic values and no critical values.
\end{definition}

\subsection{Examples}

Examples of functions in $\M_2$ are the exponential functions $\alpha e^{\beta z}$
and the tangent functions $\alpha \tan {i\beta z}=i\alpha \tanh{\beta z}$
where $\alpha,\beta$ are complex constants.

The asymptotic values for  the exponential functions
above are $\{0,\infty\}$;  the half plane $\Re{ \beta z} < 0$ is an asymptotic tract for $0$
and the half plane $\Re{ \beta z} >0$ is an asymptotic tract for infinity.
The asymptotic values for the tangent functions above are $\{\alpha i, -\alpha i\}$
and the asymptotic tract for $\alpha i$ is the half plane $\Im{ \beta z} >0$ while
the asymptotic tract for $-\alpha i$ is the half plane $\Im{ \beta z } < 0$.

\subsection{Nevanlinna's Theorem}

To find the form of the most general function in $\M_2$ we use a theorem of Nevanlinna~\cite[Chapter 10]{Nev}.   
\begin{theorem} [Nevanlinna]
Every meromorphic function $g$ with exactly $p$ asymptotic values and no critical values
has the property that its Schwarzian derivative is a polynomial of degree $p-2$.  That is
\begin{equation} \label{SCH} S(g)= \big(\frac{g''}{g'}\big)' - \frac{1}{2}\big(\frac{g''}{g'}\big)^2 = a_{p-2}z^{p-2} + \ldots a_1 z + a_0. \end{equation}
Conversely, for every polynomial $P(z)$ of degree $p-2$,  the solution to
the Schwarzian differential equation $S(g)=P(z)$ is a meromorphic function with $p$ asymptotic values and no critical values.
\end{theorem}

Here we apply this theorem to the special case when $P(z)$ is a constant.  
It is easy to check that $S(\alpha e^{\beta z} )= -\frac{1}{2}\beta^2$ and $S(\alpha \tan{\frac{i\beta}{2}}{z})=-\frac{1}{2}\beta^2$.

To find all functions in $\M_2$, let $\beta \in \CC$ be constant and consider the Schwarzian differential equation
\begin{equation}\label{eqn:Schwarzian}
 S(g)=-\beta^2/2 \end{equation}
and the related  second order linear differential equation
 \begin{equation}\label{eqn:Ricci}
 w''+\frac{1}{2}S(g)w= w'' -\frac{\beta^2}{4}  w = 0. \end{equation}
It is straightforward to check that if  $w_1, w_2$ are linearly independent solutions to  equation~(\ref{eqn:Ricci}),  then $g_{\beta}
=w_2/w_1$ is a solution to equation~(\ref{eqn:Schwarzian}). The converse is also true: if $g$ is a meromorphic function
defined in a simply connected domain in $\CC$, then two solutions $w_{1}$ and $w_{2}$ of (\ref{eqn:Ricci}) 
can be found, and furthermore, these are unique up to a common scale factor. Thus solutions to the differential equation (\ref{eqn:Ricci}) give all the solutions of the equation (\ref{eqn:Schwarzian}).  (This classical result can be found, for example, in~\cite[sec. 17.6]{Hi}). 
 
Normalizing so that $w_1(0)=1,w_1'(0)=-\beta/2, w_2(0)=1, w_2'(0)=\beta/2$ and
solving equation~(\ref{eqn:Ricci}), we have $w_1=e^{- \frac{\beta}{2} z}, w_2=e^{\frac{\beta}{2}{z}}$ as linearly independent solutions and $g_{\beta}(z)=e^{\beta z}$ and $g_{-\beta}(z)=e^{-\beta z}$ as linearly independent  solutions to  equation~(\ref{eqn:Schwarzian}).
An arbitrary solution to equation~(\ref{eqn:Schwarzian})
 then has the form \begin{equation}\label{eqn:gensoln}
 \frac{Aw_2 +Bw_1}{Cw_2+Dw_1},  \, \, A,B,C,D \in \hat{\C}, \, AD-BC= 1  \end{equation}
and its  asymptotic values are $\{A/C,B/D\}$.

\begin{remark}
The asymptotic values are distinct and omitted.
\end{remark}

\begin{remark}
If $B=C=0, AD=1, A=\sqrt{\alpha}$ we obtain the exponential family $\{\alpha e^{\beta z}\}$  with asymptotic values at $0$ and $\infty$.  If $A=-B=\sqrt{\frac{\alpha i}{2}}, \ C=D=\sqrt{-\frac{i}{2\alpha}}$ we obtain the tangent family  $\{\alpha \tan {\frac{i\beta}{2} z}\}$ whose asymptotic values $\{ \pm \alpha i\}$ are symmetric with respect to the origin.
\end{remark}

\begin{remark}\label{rmk3}
Note that in the solutions of $S(g)= - \beta^2/2$  what appears are 
$e^{\beta}$ and $e^{-\beta}$, and not $\beta$ (or $\beta^{2}$);  this creates an ambiguity about which branch 
of the logarithm of $e^{\beta}$ corresponds to a given solution of equation~(\ref{eqn:Schwarzian}).
 In section~\ref{sec:proofmt} we address this ambiguity in our situation.   We show that the topological map  we start with determines a  topological constraint which in turn, defines  the appropriate branch of the logarithm for each of the iterates in our iteration scheme.
\end{remark}

\begin{remark}
One of the basic features of the Schwarzian derivative is that it satisfies the following cocycle relation:
if $f,g$ are meromorphic functions then
$$
S(g \circ f)(z) = S(g(f)) (z) f'(z)^2 + S(f(z)).
$$
In particular, if $T$ is a M\"obius transformation,
$S(T(z))=0$ and $S(T\circ g(z) )= S((g(z))$ so that post-composing by $T$ doesn't change the Schwarzian.
\end{remark}

In our dynamical problems  the point at infinity plays a special role and the dynamics
are invariant under conjugation by an affine map.
Thus, we may assume that all the solutions have one asymptotic
value at $0$ and that they take the value $1$ at $0$.

Since this is true for $g_{\beta}(z)=e^{\beta z}$, any solution with this normalization has the form \footnote{Notice that $g_{\alpha,\beta}$ is obtained from $g_{\beta}$ by a M\"obius transformation with determinant $1$.} 
\begin{equation}\label{eqn:gennormal}
g_{\alpha,\beta}(z)= \frac{\alpha g_{\beta}(z)}{(\alpha-\frac{1}{\alpha})g_{\beta}(z) + \frac{1}{\alpha}}
\end{equation}
where $\alpha$ is an arbitrary value in $\CC^*$. The second asymptotic value is
$\lambda=\frac{\alpha}{\alpha-\frac{1}{\alpha}}$. It takes values in $\CC \setminus \{0,1\}$.
The point at infinity is an essential singularity for all these functions.

\medskip
The parameter space $\P$ for these functions is the two complex dimensional space
$$
\P=\{\alpha, \beta \in  \CC^*  \}.
$$
The parameters define a natural complex structure for the space  $\M_2$.
The subspace of {\em entire}  functions in $\M_2$ is the one dimensional subspace
of $\P$ defined by fixing $\alpha=1$ and varying $\beta$;
$$
g_{\beta}(z)=e^{\beta z}.
$$
The tangent family has symmetric asymptotic values.
Renormalized, it forms  another one dimensional subspace of $\P$.
This is defined by fixing $\alpha = \sqrt{2}$ and varying $\beta$;
$$
g_{\frac{1}{\sqrt{2}},\beta}(z)=1+\tanh{\frac{\beta}{2} z}= \frac{\sqrt{2} e^{\beta z}}{\frac{1}{\sqrt{2}}e^{\beta z}+\frac{1}{\sqrt{2}}}.
$$
These functions have asymptotic values at $\{0,2\}$ and $g_{\frac{1}{\sqrt{2}},\beta}(0)=1.$

\begin{definition}
For $g_{\alpha,\beta}(z) \in \M_2$, the set  $\Omega=\{0,\lambda\}$ of asymptotic values is the set of singular values.
The {\em post-singular set} $P_{g}$ is defined by
$$
P_{g} = \overline{ \bigcup_{n \geq 0} g^{n}(\Omega) } \cup \{\infty\}.
$$
\end{definition}

Note that we include the point at infinity separately in $P_g$ because
whether or not it is an asymptotic value, it is an essential singularity
and its forward orbit is not defined.
The asymptotic values are in $P_g$ and,
since $0$ and $\lambda$ are omitted and $g_{\alpha,\beta}(0) =1 \in P_g$, $\#P_g \geq 3$.

\section{The Space $\T2$ }
We now want to consider the topological structure of functions in $\M_2$
and define $\T2$ to be the set of maps with the same topology.

\begin{definition}
Let $X$ be a simply connected open surface and let $S^2={\mathbb R}^{2}\cup \{\infty\}
 $ be the 2-sphere.
Let $f_{a,b}:X \rightarrow S^2 \setminus \{a,b\}$ be an unbranched covering map;
that is, a universal covering map.
We say the pair $(X,f_{a,b})$ is equivalent
to the pair $(Y, f_{c,d})$ if and only if there is a homeomorphism $h:X \rightarrow Y$
such that $f_{c,d} \circ h= f_{a,b}$.
An equivalence class of such classes is called a {\em 2-asymptotic value map }
and the space of these pairs is denoted by $\T2$.
\end{definition}

Let $(X,f_{a,b})$ be a representative of a map in $\T2$.
By abuse of notation, we will  often suppress the dependence
on the equivalence class and identify $X$ with ${\mathbb R}^{2}=S^2 \setminus \{\infty\}$
and refer to $f_{a,b}$ as an element of $\T2$.

By definition $f_{a,b}$ is a local homeomorphism and satisfies the following conditions:

For $v=a$ or $v=b$,  let $U_v \subset X$ be a neighborhood of $v$ whose boundary is a simple closed curve
that separates $a$ from $b$ and contains $v$ in its interior.

\begin{enumerate}
\item $f_{a,b}^{-1}(U_v \setminus \{v\})$ is connected and simply connected.
\item The restriction $f_{a,b}: f_{a,b}^{-1}(U_v \setminus \{v\}) \rightarrow U_v \setminus \{v\}$ is a regular covering of a punctured topological disk whose degree is  infinite.
\item   $f^{-1}( \partial U_v)$ is an open curve extending to infinity in both directions.
\end{enumerate}
In analogy with meromorphic functions with isolated singularities we say

\begin{definition}
$v$ is called a {\em logarithmic singularity} of $f_{a,b}^{-1}$ or, equivalently, an    {\em asymptotic value} of $f_{a,b}$.
The domain $V_v=f_{a,b}^{-1}(U_v \setminus \{v\})$ is called an {\em asymptotic tract for $v$}.
\end{definition}

\begin{definition}
{\em  $\Omega_f=\{a,b\}$ is the set of singular values of $f_{a,b}$}.
\end{definition}

Endow $S^2$ with the standard complex structure so that it is identified with $\hat\CC$.
By the classical uniformization theorem, for any pair $(X,f_{a,b})$,
there is a  map $\pi: \CC \rightarrow  X$  such that $g_{a,b} = f_{a,b} \circ \pi $ is meromorphic.
It is called {\em the meromorphic function associated to $f_{a,b}$}.

By Nevanlinna's  theorem $S(g(z))$ is constant and moreover,

\begin{proposition}\label{prop2}
If  $g(z) \in \M_2 $ with $\Omega_g=\{a,b\}$ then $g(z)=g_{a,b}(z)  \in  \T2$ and, conversely,
if $g_{a,b} \in \T2$ is meromorphic then $g_{a,b} \in \M_2$.
\end{proposition}

\begin{proof}
Any $g(z) \in \M_2$ is a universal cover $g:\CC \rightarrow \hat\CC \setminus \Omega_g$
and so belongs to $\T2$.
Conversely, if $g_{a,b} \in \T2$, it is meromorphic and its only singular values are the omitted values $\{a,b\}$;
it is thus in $\M_2$.
\end{proof}

We define the post-singular set for maps in $\T2$ just as we did for functions in $\M_2$.

 \begin{definition}
 For $f=f_{a,b} \in \T2$, the {\em post-singular set} $P_{f}$ is defined by
 $$
 P_{f} = \overline{ \bigcup_{n \geq 0} f^{n}(\Omega_{f}) } \cup \{\infty\}
 $$
\end{definition}

Note that under the identification of $S^2$ with the Riemann sphere and $X$ with the complex plane,
$S^2 \setminus X$ is the point at infinity and it has no forward orbit although it may be an asymptotic value.
We therefore include it in $P_f$.

Conjugation  of $f_{a,b}$ by an affine transformation $T$ results in another map in $\T2$.
In what follows, therefore, we  always assume $X$ is the Euclidean plane ${\mathbb R}^{2}$,  one asymptotic value is $a=0$ and the second asymptotic value is  $b=\lambda$ and we normalize so that $f(0)=1$.

\medskip
\begin{definition}
We call $f\in \T2$  {\em post-singularly finite} if $\#(P_{f})<\infty$.
\end{definition}

\section{Combinatorial Equivalence}
\label{sec:combequiv}

\begin{definition}~\label{combeq}
Suppose $f$ and $g$ are a pair of  post-singularly  finite maps  either in ${\mathcal TE}_{p,q}$ or in  $\T2$. We say that $f$ and $g$ are {\em combinatorially equivalent}  if  there are two 
homeomorphisms $\phi$ and $\psi$ of the sphere $S^{2}={\mathbb R}^{2}\cup\{\infty\}$ fixing $0$ and $\infty$ such that $\phi\circ f=g\circ \psi$ on ${\mathbb R}^{2}$ and if, in addition,  $\phi^{-1}\circ \psi$ is isotopic to the identity of $S^{2}$ rel $P_{f}$;  that is, $\phi|P_{f}=\psi|P_{f}$. 
%~\footnote{Note that {\em topological equivalence} is NOT the same as {\em topological conjugacy} since $\phi$ and $\psi$ are two two distinct homeomorphisms. However, its  restriction  to  $P_{f}$  is a topological conjugacy since $\phi|P_{f}=\psi|P_{f}$}. 
\end{definition}

The commutative diagram for the above definition is
\begin{equation*}
\xymatrix{\mathbb{R}^2 \ar[d]^f\ar[r]^{\psi} & \mathbb{R}^2\ar[d]^{g}\\
\mathbb{R}^2\ar[r]^\phi & \mathbb{R}^2}
\end{equation*}
The isotopy condition says that $P_{g}= \phi(P_{f})$.

Consider
${\mathbb R}^{2}\cup \{\infty\}$ equipped with the standard conformal structure as the Riemann sphere and let $f$ be a map from $X \subset {\mathbb C}$ into $\hat{\CC}$.  We say $f$ is {\em locally $K$-quasiconformal} for some $K>1$ if for any $z\in  \hat{\C} \setminus (\Omega_{f}\cup \{0,1\})$  there is a neighborhood $U$ of $z$ such that $f: U\to f(U)$ is $K$-quasiconformal.  Since the maps $f$ we are working with are isotopies rel a finite set, the following lemma is standard so we omit the proof. 

\begin{lemma}
Any post-singularly finite $f\in {\TE}_{p,q}$  (or $f\in \T2$) is combinatorially equivalent to some locally $K$-quasiconformal map $g\in {\TE}_{p,q}$ (or $g\in \T2$).
\end{lemma}

Thus without loss of generality, in the rest of the paper, we will assume that any post-singularly finite $f \in {\TE}_{p,q}$ (or $f\in \T2$) is locally $K$-quasiconformal for some $K\ge 1$.

\section{Teichm\"uller Space $T_{f}$}
\label{sec:teich sp}

Recall that we denote $\mathbb{R}^2 \cup \{\infty\}$ equipped with the standard conformal structure by $\widehat{\mathbb{C}}$.  Let ${\mathcal M}=\{ \mu \in L^{\infty} (\widehat{\mathbb C})\;|\; \|\mu\|_{\infty}<1\}$ be the unit ball in the space of all measurable functions on the Riemann sphere. Each element $\mu\in {\mathcal M}$ is called a Beltrami coefficient.
For each Beltrami coefficient $\mu$, the Beltrami equation
$$
w_{\overline{z}}=\mu w_{z}
$$
has a unique quasiconformal  solution $w^{\mu}$ which maps $\hat{\mathbb C}$  to itself fixing $0,1, \infty$.
Moreover, $w^{\mu}$ depends holomorphically  on $\mu$.

Let $f$ be a post-singularly finite map in ${\mathcal TE}_{p,q}$  ( or $ \T2$) with post-singular set $P_f$.    The Teichm\"uller space $T(\hat{\mathbb C}, P_f)$ is defined as follows.  Given Beltrami differentials   $\mu, \nu \in {\mathcal M}$  we say that      $\mu$ and $\nu$ are equivalent in $\mathcal M$,  and denote this by  $\mu\sim \nu$, if $(w^{\nu})^{-1} \circ w^{\mu}$ is isotopic to the identity map of $\widehat{\mathbb C}$ rel $P_f$. The equivalence  class of $\mu$ under $\sim$ is denoted by $[\mu]$.
We set
$$
T_f=T(\hat{\mathbb C}, P_f)= \mathcal M/ \sim.
$$

It is easy to see that $T_{f}$ is a finite-dimensional complex manifold and is equivalent to the classical Teichm\"uller space $Teich(\hat{\mathbb C}\setminus P_{f})$ of Riemann surfaces with  basepoint $\hat{\mathbb C}\setminus P_{f}$. Therefore, the Teichm\"uller distance $d_{T}$ and the Kobayashi distance $d_{K}$ on $T_{f}$ coincide (see e.g. ~\cite{GJW}).

\section{Induced Holomorphic Map $\sigma_{f}$}
\label{sec:Thurston map}

For any post-singularly finite $f$ in ${\mathcal TE}_{p,q}$ (or $T2$), there is an induced  map $\sigma= \sigma_{f}$ from $T_{f}$ into itself given by
$$
\sigma([\mu]) =[f^{*}\mu],
$$
where
\begin{equation}~\label{pullbackformula}
f^{*}\mu(z) = \frac{\mu_f(z) + \mu_f((f(z))
\theta(z)}{1 + \overline{\mu_f (z)} \mu_f(f(z)) \theta(z)}, \;\; \mu_{f} (z) =\frac{f_{\bar{z}}}{f_{z}},  \;\; \theta(z) =\frac{\bar{f}_{z}}{f_{z}}.
\end{equation}
It is a holomorphic map, so it contracts the the Kobayashi distance $d_{K}$.  By the final paragraph in the last section,  this means it is a contraction in the Teichm\"uller distance $d_{T}$. Thus we have that

\vspace*{5pt}
\begin{lemma}~\label{contractive}
For any two points $\tau$ and $\tilde\tau$ in $T_{f}$,
$$
d_{T}(\sigma(\tau), \sigma(\tilde\tau))\leq d_{T}(\tau, \tilde\tau).
$$
\end{lemma}

 The next lemma follows directly from the definitions.
\vspace*{5pt}
\begin{lemma}~\label{fixedpt}
A  post-singularly finite $f$ in ${\TE}_{p, q}$ (or $\T2$) is combinatorially equivalent to a $(p,q)$-exponential map $E=Pe^{Q}$ (or a meromorphic map in $\M2$) if and only if  $\sigma$ has a fixed point in $T_{f}$.
\end{lemma}

\begin{proof}
Suppose $\sigma$ has a fixed point $\tau=[\mu]$, that is, $\sigma (\tau)=[f^{*}\mu]=\tau=[\mu]$. This implies that 
$w^{\mu}$ and $w^{f^{*}\mu}$ are isotopy rel $P_{f}$. Using Formula (\ref{pullbackformula}), one can check that 
$$
E= w^{\mu}\circ f\circ (w^{f^{*}\mu})^{-1}
$$
is holomorphic. This says that $f$ is combinatorially equivalent to $E(z) = P(z) e^{Q(z)}$. This proves the ``if'' part.

For the ``only if'' part, suppose there are two quasiconformal homeomorphisms $\phi$ and $\psi$  as in Definition~\ref{combeq} such that $\phi$ and $\psi$ are isotopic rel $P_{f}$ and $E=\phi \circ f\circ \psi^{-1}=Pe^{Q}$. Let $\mu$ be the Beltrami coefficient of $\phi$. Then the last equality implies that $f^{*}\mu$ is the Beltrami coefficient of $\psi$. Thus
$w^{\mu}$ and $w^{f^{*}\mu}$ are isotopic rel $P_{f}$. This implies that $\sigma (\tau)=\tau$ as required. 
\end{proof}

\section{Bounded Geometry}
\label{sec:bounded geometry}

For any $\tau_{0}\in T_{f}$, let $\tau_{n}=\sigma^{n}(\tau_{0})$, $n\geq 1$. The iteration sequence $\tau_{n}=[\mu_{n}]$ determines a sequence of finite subsets
$$
P_{f,n} = w^{\mu_{n}}(P_{f}), \quad n=0, 1, 2, \cdots.
$$
Since all $w^{\mu_{n}}$ fix $0, 1, \infty$, it follows that $0, 1, \infty\in P_{f,n}$.   Note that $P_{f,n}$ depends only on $\tau_{n}$ and not $\mu_{n} $ by the Teichm\"uller equivalence relation.

\begin{definition}[Spherical Version]
We say $f$ has {\em bounded geometry} if there is a constant $b>0$ and a point $\tau_{0}\in T_{f}$ such that
$$
d_{sp} (p_{n},q_{n}) \geq b
$$
for $p_{n}, q_{n}\in  P_{f,n}$ and $n\geq 0$. Here
$$
d_{sp}(z,z')= \frac{|z-z'|}{\sqrt{1+|z|^{2}}\sqrt{1+|z'|^{2}}}
$$
is the spherical distance on $\hat{\mathbb C}$.
\end{definition}

Note that $d_{sp}(z, \infty) = \frac{|z|}{\sqrt{1+|z|^2}}$.
Away from infinity the spherical metric and Euclidean metric are equivalent.
Precisely, for any bounded  $S \subset \mathbb C$,
there is a constant $C>0$ which depends only on $S$ such that
$$
C^{-1} d_{sp}(x,y) \leq |x-y|\leq Cd_{sp}(x,y)\quad  \forall  x,y \in S.
$$

Consider the hyperbolic Riemann surface $R=\hat{\mathbb C}\setminus P_{f}$
equipped with the standard complex structure as the basepoint $\tau_{0}=[0]\in T_{f}$.
A point $\tau$ in $T_{f}$ defines another complex structure $\tau$ on $R$.
Denote by $R_{\tau}$  the hyperbolic Riemann surface  $R$ equipped with the complex structure $\tau$.

A simple closed curve $\gamma\subset R$ is called {\it non-peripheral} if each component of $\hat{\mathbb C}\setminus \gamma$ contains at least two points
of $P_{f}$. Let $\gamma$ be a non-peripheral simple closed curve in $R$.
For any $\tau=[\mu]\in T_{f}$, let $l_{\tau}(\gamma)$ be the hyperbolic length of the unique closed geodesic homotopic to $\gamma$ in $R_{\tau}$.
The bounded geometry property can be stated in terms of  hyperbolic geometry as follows.

\begin{definition}[Hyperbolic version]
We say $f$ has {\em bounded geometry}  if there is a constant $a>0$ and a point $\tau_{0}\in T_{f}$ such that
 $l_{\tau_{n}}(\gamma)\geq a$ for all $n\geq 0$ and all non-peripheral simple closed curves $\gamma$ in $R$.
\end{definition}

The above definitions of bounded geometry are equivalent because of the following lemma and the fact that we have normalized so that $0,1,\infty$ always belong to $P_f$.

\vspace*{5pt}
\begin{lemma}~\label{sg} Consider the hyperbolic Riemann surface
$\hat{\mathbb C}\setminus X$, where $X$ is a finite subset of $\hat\CC$ such that $0, 1, \infty \in X$, equipped with the standard complex structure.
Let $a>0$ be a constant. If every simple closed geodesic in $\hat{\mathbb C}
\setminus X$ has hyperbolic length greater than $a$, then the
spherical distance between any two distinct points in $X$ is bounded below by a bound $b>0$ which depends only on $a$ and $m=\#(X)$.
\end{lemma}

The reader can refer to~\cite{CJK} for a detailed proof. So we omit it in this paper.

\section{The Proof of Necessity}
\label{sec:main result}

Our  theorems  have two parts: the necessity and sufficiency of the bounded geometry condition. The necessity is relatively easy and can be proved once for all cases together.  
We prove the following statement. 

\vspace*{5pt}
\begin{theorem}[Necessity]~\label{necc}
If a post-singularly finite map $f\in {\mathcal TE}_{p,q}$ (or $ \T2$) is combinatorially equivalent to a $(p,q)$-exponential map $E=Pe^{Q}\in \E_{p,q}$ (or a meromorphic map $g\in \M2$), then $f$ has bounded geometry.
\end{theorem}

\begin{proof}
If $f$ is combinatorially equivalent to $E=Pe^{Q}\in \E_{p,q}$ (or $g\in \M2$), then
$\sigma$ has a  fixed point $\tau_{0}$ so that $\tau_{n}=\tau_{0}$ for all $n$.  The complex structure on $\hat{\mathbb C} \setminus P_f$  defined by $\tau_0$ induces  a hyperbolic metric on it.   The shortest closed geodesic in this metric gives a lower bound on the lengths  of all geodesics  so that  $f$ satisfies the hyperbolic definition of bounded geometry.
\end{proof}

\section{Sufficiency under Compactness}
\label{sec:suff}

The proof of the sufficiency of bounded geometry  in our theorems is more complicated and needs some preparatory material.
There are two parts:  one is a compactness argument and the other is a fixed point argument.
Once one has compactness, the proof of the fixed point argument  is quite standard (see~\cite{Ji}) and works for any $f\in {\mathcal TE}_{p,q}$ and any $f\in \T2$. 
This is the content of Theorem~\ref{main1}  whose proof we give  in this section. We only give the details  for $f\in {\mathcal TE}_{p,q}$. For $f\in \T2$, the proof is similar  but uses  different notation (see~\cite{CJK}).

The normalized functions in ${\mathcal E}_{p,q}$ are determined by the $p+q+1$ coefficients of the polynomials $P$ and $Q$.
This identification defines an embedding into $\CC^{p+q+1}$ and hence a topology on ${\mathcal E}_{p,q}$.

Given $f\in\TE_{p,q}$ and given any $\tau_{0}=[\mu_{0}]\in T_f$, let $\tau_{n}=\sigma^{n} (\tau_{0}) =[\mu_{n}]$
be the sequence generated by $\sigma$. Let $w^{\mu_{n}}$ be the normalized quasiconformal map with Beltrami coefficient $\mu_{n}$.
Then
$$
E_{n} = w^{\mu_{n}}\circ f\circ (w^{\mu_{n+1}})^{-1}\in {\mathcal E}_{p, q}
$$
since it preserves
$\mu_0$ and hence is holomorphic.
This gives a sequence $\{E_{n}\}_{n=0}^{\infty}$ of maps in ${\mathcal E}_{p,q}$ and a sequence of subsets
$P_{f,n}=w^{\mu_{n}} (P_{f})$.
Note that $P_{f,n}$ is not, in general, the post-singular set $P_{E_{n}}$ of $E_{n}$.

\medskip
\noindent {\bf The compactness condition.}  We say $f$ satisfies the compactness condition if the sequence $\{ E_n \}_{n=1}^{\infty}$ generated  in the Thurston iteration scheme is contained in a compact subset of $ \E_{p,q}$.

\medskip
From a conceptual point of view, the compactness condition is very natural and simple. From a technical point of view, however, it is not at all obvious.  We give a detailed proof showing how to get a fixed point assuming both bounded geometry and compactness. 

Suppose $f$ is a post-singularly finite topological exponential map in ${\mathcal TE}_{p,q}$.
For any $\tau=[\mu]\in T_{f}$, let $T_{\tau} $ and $T^{*}_{\tau} $ be the tangent space and the cotangent space of $T_{f}$ at $\tau$ respectively.  Let $w^{\mu}$ be the corresponding normalized quasiconformal map fixing $0,1, \infty$.
Then $T^{*}_{\tau} $ coincides with the space ${\mathcal Q}_{\mu}$ of integrable meromorphic quadratic differentials $q=\phi(z) dz^{2}$.  Integrablility means that  the norm of $q$,  defined by
$$
||q|| =\int_{\hat{\mathbb C}} |\phi(z)| dzd\overline{z}
$$
is finite.  This condition implies that  the  poles of $q$  must occur at points of $w^{\mu} (P_{f})$ and that these poles are simple.

Set $\tilde{\tau}=\sigma(\tau)=[\tilde{\mu}]$ and denote by $w^{\mu}$ and $w^{\tilde{\mu}}$   the corresponding  normalized quasiconformal maps. We have the following commutative
diagram:
$$
\begin{array}{ccc} \hat{\mathbb C}\setminus f^{-1}(P_{f}) & {\buildrel w^{\tilde{{\mu}}} \over
\longrightarrow} & \hat{\mathbb C}\setminus w^{\tilde{{\mu}}} (f^{-1}(P_{f}))\cr \downarrow f
&&\downarrow E_{\mu,\tilde{\mu}}\cr \hat{\mathbb C}\setminus P_{f}& {\buildrel w^{\mu}
\over \longrightarrow} & \hat{\mathbb C}\setminus w^{\mu}(P_{f}).
\end{array}
$$
Note that in the diagram, by abuse of notation, we write  $f^{-1}(P_{f})$ for  $f^{-1}(P_{f}\setminus \{\infty\})\cup\{\infty\}$.  Since by definition  $\tilde{\mu} = f^*\mu$, the map
  $E=E_{\mu,\tilde{\mu}}=  w^{\mu} \circ f \circ  (w^{\tilde{\mu}})^{-1}$ defined on ${\mathbb C}$
is analytic. By Theorem~\ref{topexp}, $E_{\mu,\tilde{\mu}}=P_{\tau, \tilde{\tau}}e^{Q_{\tau,\tilde{\tau}}}$ for a pair of polynomials $P=P_{\tau, \tilde{\tau}}$ and $Q=Q_{\tau,\tilde{\tau}}$ of respective degrees $p$ and $q$.

Let $\sigma_{*}: T_{\tau} \to T_{\tilde{\tau}} $ and  $\sigma^{*}: T_{\tilde{\tau}}^* \to T_{\tau}^* $ 
be the tangent and co-tangent maps of $\sigma$, respectively. Take a co-tangent vector 
$\tilde{q}=\tilde{\phi} (w) dw^{2}$ in $T^{*}_{\tilde{\tau}}$. Let $q=\sigma^{*} \tilde{q}$ be the corresponding co-tangent vector in $T^{*}_{\tau}$.   Then $q$
is also the push-forward integrable quadratic differential of $\tilde{q}$ by $E$
\begin{equation}~\label{push}
q=E_{*}\tilde{q}=\phi(z) dz^2.
\end{equation}

   To see this, recall from section~\ref{sec:Tpq} that $E$, and a choice of curves $L_i$ from the branch points, determine a finite set of domains $W_i$ on which $E$ is an unbranched covering to a domain homeomorphic to $\CC^*$.  Since $E$ restricted to each $W_i$ is either a topological model for $e^z$ or $z^k$, we may divide each $W_i$ into a collection of  fundamental domains on which $E$ is bijective. Therefore the coefficient $\phi (z)$ of $q$ is given by the formula
\begin{equation}~\label{pushforwardformula}
\phi(z)= ({\mathcal L}\tilde{\phi}) (z) =\sum_{E(w) = z}
\frac{\tilde{\phi} (w)}{(E'(w))^{2}} = \frac{1}{z^{2}} \sum_{E(w)=z} \frac{\tilde{\phi} (w)}{(\frac{P'(w)}{P(w)} +Q'(w))^{2}}
\end{equation}
Here ${\mathcal L}$ is called a transfer operator in thermodynamical formalism.
Following some standard calculations (see e.g.~\cite{Ji}) on transfer operators, we have, 

\begin{equation}~\label{cotangcontracting}
||q||<||\tilde{q}||.
\end{equation}

\begin{remark}
The main point in the calculations is to note that $E$ has infinite degree   and   $q$ has finitely many poles.  If  there were a $\tilde{q}$ with  $||q||=||\tilde{q}||\not= 0$ and poles comprising a set $Z$, then
the poles of $q$ would be contained in the set $E(Z)\cup {\mathcal V}_{E}$, where ${\mathcal V}_{E}$ is the set of critical values of $E$.
Thus, by formula~(\ref{pushforwardformula}),
$$
E^{*} q =\phi (E(w)) dw^2 = n \tilde{q} (w),
$$
where $n$ is the degree of $E$. Furthermore,
$$
E^{-1} (E(Z)\cup {\mathcal{V}}_{E}) \subseteq Z \cup {\Omega}_{E}.
$$
Since  $n$ is infinite,  the last inclusion formula can not hold  because the left hand side is infinite and the right hand side is finite.
\end{remark}

An  immediate corollary  of inequality (\ref{cotangcontracting}) is
\begin{corollary}~\label{strongcon}
For any two points $\tau$ and $\tilde{\tau}$ in $T_{f}$,
$$
d_{T}\Big(\sigma(\tau), \sigma(\tilde{\tau})\Big)< d_{T}(\tau, \tilde{\tau}).
$$
\end{corollary}

 In addition,  inequality (\ref{cotangcontracting}) also implies uniqueness. 

\begin{corollary}
If $\sigma$ has a fixed point in $T_{f}$, then this fixed point must be unique. This is equivalent to saying  that
 a post-singularly finite $f$ in ${\TE}_{p, q}$  is combinatorially equivalent to at most one  $(p,q)$-exponential map $E=Pe^{Q}$.
\end{corollary}

We can now finish the proof of the sufficiency in Theorem~\ref{main1}.

\begin{proof}[Proof of Theorem~\ref{main1}] 
Suppose $f\in {\TE}_{p,q}$ has both bounded geometry and compactness. Suppose $\tau_{0}=[\mu_{0}]$ satisfies  the bounded geometry condition and $\mu_{n}$ is defined by $\sigma^{n} (\tau_{0}) =[\mu_{n}]$. 
Recall that the map defined by
\begin{equation}\label{thu iter}
E_{n}=w^{\mu_{n}}\circ f\circ (w^{\mu_{n+1}})^{-1}
\end{equation}
is a $(p,q)$-exponential map.

If $q=0$, $E_n$ is a polynomial and the theorem follows from the arguments given in~\cite{CJ} and \cite{DH}.
Note that if $P_f = \{0,1,\infty\}$, then $f$ is a universal covering map of $\mathbb{C}^*$ and is therefore combinatorially equivalent to $e^{2 \pi i z}$.
Thus in the following argument, we assume that $\#(P_f) \geq 4$. Then, given our normalization conventions and the bounded geometry hypothesis we see that
the functions  $E_{n}$, $n=0, 1, \ldots$ satisfy the following conditions:
\begin{itemize}
\item[1)] $m=\#(w^{\mu_{n}}(P_f))\geq 4$ is fixed.
\item[2)] $0, 1, \infty \in w^{\mu_{n}}(P_f)$.
\item[3)] $\Omega_{E_{n}}\cup \{ 0, 1, \infty\} \subseteq
E_{n}^{-1}(w^{\mu_{n}}(P_f))$.
\item[4)] there is a $b>0$ such that $d_{sp}(p_{n}, q_{n}) \geq b$ for any $p_{n}, q_{n}\in w^{\mu_{n}}(P_f)$.
\end{itemize}

As a consequence of the compactness, we have that in the sequence $\{ E_{n}\}_{n=1}^{\infty}$, there is a subsequence $\{ E_{n_{i}}\}_{i=1}^{\infty}$ converging to a map $E=Pe^{Q} \in \E_{p,q}$ where $P$ and $Q$ are polynomials of degrees $p$ and $q$ respectively.

Any integrable quadratic differential $q_{n} \in T^{*}_{\tau_{n}}$ has, at worst, simple poles in the finite set  $P_{f,n}= w^{\mu_{n}}(P_f)$.
Since $T^{*}_{\tau_{n}}$ is a finite dimensional linear space, there is a quadratic differential $q_{n, max}\in T^{*}_{\tau_{n}}$
with $\| q_{n,max}\|=1$ such that
$$
0 \leq a_{n}=\sup_{||q_{n}||=1} \|(E_{n})_{*}q_{n}|| = \|(E_{n})_{*}q_{n,max}\| <1.
$$  
Moreover, by the bounded geometry condition any simple poles of $\{q_{n, max}\}_{n=1}^{\infty}$ lie in a
  compact set and hence these quadratic differentials lie in a compact subset of the  space of quadratic differentials on $\hat{\mathbb C}$ with, at worst, simples poles at $m=\#(P_{f})$ points.

Let
$$
a_{\tau_{0}} =\sup_{n\geq 0} a_{n}.
$$
Let  $\{n_{i}\}$  be a sequence of  integers such that the subsequence $a_{n_{i}}\to a_{\tau_{0}}$ as $i\to \infty$.
By compactness, $\{E_{n_{i}}\}_{i=0}^{\infty}$ has
a convergent subsequence, (for which we  use the same notation)
that converges to a holomorphic  map $E \in \E_{p,q}$.
Taking a further subsequence if necessary, we obtain a convergent sequence of sets $P_{n_{i},\tau_{0}} =w^{\mu_{n_{i}}}(P_{f})$  with limit set $X$.
By bounded geometry, $\#(X)=\#(P_{f})$ and  $d_{sp}(x, y) \geq b$ for any $x,y\in X$.
Thus we can find a   subsequence $\{ q_{n_{i}, max}\}$ converging to an integrable quadratic differential $q$ of norm $1$  whose only poles lie in $X$ and are simple.
Now by inequality (\ref{cotangcontracting}), we have that
$$
a_{\tau_{0}} = ||E_{*} q|| <1.
$$

Thus we have proved  that there is an  $0< a_{\tau_{0}}< 1$, depending only on $b$ and $f$, such that
$$
\|\sigma_{*}\| \le
\|\sigma^{*}\| \le a_{\tau_{0}}.
$$
Let $l_{0}$ be a curve connecting $\tau_{0}$ and $\tau_{1}$ in $T_{f}$ and set $l_{n}=\sigma_{f}^{n}(l_{0})$ for $n\geq 1$. Then $l=\cup_{n=0}^{\infty}l_{n}$ is a curve in $T_f$
 connecting all the points $\{\tau_{n}\}_{n=0}^{\infty}$. For each point $\tilde{\tau}_{0}\in l_{0}$, we have $a_{\tilde{\tau}_{0}} <1$. Taking the maximum  gives  a uniform $a<1$ for all points in $l_0$.  Since $\sigma$ is holomorphic, $a$ is an upper bound for all points in $l$.  Therefore,

$$
d_{T} (\tau_{n+1}, \tau_{n}) \leq a \, d_{T}(\tau_{n}, \tau_{n-1})
$$
for all $n\geq 1$.
Hence, $\{ \tau_{n}\}_{n=0}^{\infty}$ is a convergent sequence with a unique limit point $\tau_{\infty}$ in $T_{f}$ and $\tau_{\infty}$  is
a fixed point of $\sigma$. This together with Lemma~\ref{fixedpt} completes the proof of  sufficiency in Theorem~\ref{main1}.
\end{proof}

From our proof of Theorem~\ref{necc},  the final step in the  proofs of Theorems~\ref{main2} and~\ref{main3} is to show   that 
the compactness condition follows from bounded geometry.  This is in contrast to  the case of rational maps (see~\cite{Ji}) where the bounded geometry condition always guarantees the compactness condition holds. In the case of $(p,q)$-exponential maps, 
the bounded geometry condition must be combined with some topological constraints to guarantee the compactness. For Theorems~\ref{main2} and ~\ref{main3} we define  topological constraints that, together with the bounded geometry condition, control the sizes of the fundamental domains coming from the decomposition of $\mathbb C$ described Definition~\ref{topexpdef} so that they are neither too small nor too big. Thus, before we prove  bounded geometry implies the  compactness condition holds, we will describe these  topological constraints for the maps in these theorems.  

\section{Topological Constraints.}
\label{sec:proofmt}
In section~\ref{sec:Tpq} we defined two different  normalizations   for functions in $\TE_{p,q}$ that depend on whether or not $0$ is a fixed point of the map.  The topological constraints 
for post-singularly finite maps also follow this dichotomy.    

\subsection{For $f\in \TE_{0,1}$ satisfying the hypotheses of Theorem~\ref{main3}.}
\label{top1}
Any such $f$ has no branch points  so $P_{f} =\cup_{k\geq 0} f^{k} (0) \cup\{\infty\}$.  Since $0$ is omitted and $P_{f}$ is  finite, the orbit of $0$ is pre-periodic. Let $c_{k}=f^{k}(0)$ for $k\geq 0$. By the pre-periodicity, there are a minimal integer $k_{1}\geq 0$ and a minimal integer $l\geq 1$ such that 
$f^{l}(c_{k_{1}+1}) =   c_{k_{1}+1}$. This says that 
$$
\{ c_{k_{1}+1}, \ldots, c_{k_{1}+l}\}
$$
is a periodic orbit of period $l$. Let $k_{2} =k_{1}+l$.  
  Let $\gamma$ be a continuous curve 
connecting $c_{k_{1}}$ and $c_{k_{2}}$ in 
${\mathbb R}^{2}$ disjoint from $P_f$,  except for its endpoints.  Because 
$$
f(c_{k_{1}})=f(c_{k_{2}})=c_{k_{1}+1},
$$ 
the image curve $\delta= f(\gamma)$  is a closed curve.  

\medskip
\subsection{For $f\in \TE_{p,1}$ satisfying the hypotheses of Theorem~\ref{main2}.}
\label{top2}
Any such $f$ has exactly one non-zero simple branch point which we denote by $c$; 
$0$ is the only other branch point and it  has multiplicity $p-1$.  Then $f(0)=0$ and by our normalization, $f(c)=1$.  In this case
$$
P_{f} =\cup_{k\geq 1} f^{k} (c) \cup\{0, \infty\}.
$$
Again by the hypothesis of Theorem~\ref{main2},  $P_{f}$ is finite. Set $c_{k}=f^{k}(c)$ for $k\geq 0$. 

Suppose $c$ is not periodic. As above, there are an minimal integer $k_{1}\geq 0$ and a minimal integer $l\geq 1$ such that 
$f^{l}(c_{k_{1}+1}) =   c_{k_{1}+1}$.  Again,  
$$
\{ c_{k_{1}+1}, \ldots, c_{k_{1}+l}\}
$$
is a periodic orbit of period $l$. Let $k_{2} =k_{1}+l$.  

As above, let $\gamma$ be a continuous curve 
connecting $c_{k_{1}}$ and $c_{k_{2}}$ in 
${\mathbb R}^{2}$ disjoint from $P_f$,  except for its endpoints. Since 
$$
f(c_{k_{1}})=f(c_{k_{2}})=c_{k_{1}+1},
$$ 
the image curve $\delta= f(\gamma)$  is a closed curve.  

\subsection{Winding number}
\label{winding}
In each of the above cases, the {\em winding number} $\eta$ of the closed curve $\delta= f(\gamma)$ about $0$ essentially counts the number of fundamental domains 
between $c_{k_{1}}$ and $c_{k_{2}}$  and defines the ``distance'' between these fundamental domains.    
The following lemma is a crucial to proving that the compactness condition holds for each type of function in Theorem~\ref{main2} and Theorem~\ref{main3}.   

 \medskip 
\begin{lemma}\label{winding1}
The winding number $\eta$ is  does not change  under the Thurston iteration procedure.
\end{lemma}
\begin{proof}

Given $\tau_{0}=[\mu_{0}]\in T_f$, let $\tau_{n}=\sigma^{n} (\tau_{0}) =[\mu_{n}]$
be the sequence generated by $\sigma$. Let $w^{\mu_{n}}$ be the normalized quasiconformal map with Beltrami coefficient $\mu_{n}$.
Then in either of the above situations, 
$$
E_{n} = w^{\mu_{n}}\circ f\circ (w^{\mu_{n+1}})^{-1}\in {\mathcal E}_{p, 1} 
$$
since it  preserves $\mu_0$ and  is holomorphic. See the following diagram.
$$
\begin{array}{ccc} \hat{\mathbb C}& {\buildrel w^{\mu_{n+1}} \over
\longrightarrow} & \hat{\mathbb C}\cr
\downarrow f &&\downarrow E_{n}\cr
\hat{\mathbb C}& {\buildrel w^{\mu_{n}}
\over \longrightarrow} & \hat{\mathbb C}.
\end{array}
$$
Let $c_{k, n} =w^{\mu_{n}}(c_{k})$. 
The continuous curve
$$
\gamma_{n+1}=w^{\mu_{n+1}} (\gamma)
$$
goes from  $c_{k_{1}, n+1}$ to $c_{k_{2}, n+1}$.  
The image curve is 
$$
\delta_{n}=E_{n} (\gamma_{n+1}) =   w^{\mu_{n}}(f((w^{\mu_{n+1}})^{-1} (\gamma_{n+1})))=w^{\mu_{n}} (f (\gamma)) = w^{\mu_{n}} (\delta).
$$
Note that $w^{\mu_{n}}$ is a homeomorphism that fixes $0,1,\infty$. Thus $\delta_{n}$ is a closed curve  through the point $c_{k_{1}+1, n}= w^{\mu_{n}} (c_{k_{1}+1})$ and it has    winding number $\eta$ around $0$.
\end{proof}
 
\subsection{For $f\in \T2$ satisfying the hypotheses of Theorem~\ref{main2}.}
\label{top3}

Suppose $f \in \T2$.  Recall $\Omega_{f}=\{0, \lambda\}$ is the set of asymptotic values of $f$ and that we have normalized so that $f(0)=1$. 
Suppose that this $f$ is post-singularly finite;  that is, $P_{f}$ is finite so that 
the orbits $\{c_{k}=f^{k} (0)\}_{k=0}^{\infty}$ and $\{c_{k}'=f^{k} (\lambda)\}_{k=0}^{\infty}$ are both finite, and thus, preperiodic.  Note that neither can be  periodic because the asymptotic values are omitted.    Consider the orbit of $0$.  Preperiodicity  means there are  minimal integers $k_{1}\geq 0$ and  $l\geq 1$ such that $f^{l}(c_{k_{1}+1}) =   c_{k_{1}+1}$.  That is, 
$$
\{ c_{k_{1}+1}, \ldots, c_{k_{1}+l}\}
$$
is a periodic orbit of period $l$. Set $k_{2} =k_{1}+l$.  

Let $\gamma$ be a continuous curve connecting $c_{k_{1}}$ to $c_{k_{2}}$ in 
${\mathbb R}^{2}$ which is disjoint from $P_f$, except at its endpoints. 
Because $f(c_{k_{1}})=f(c_{k_{2}})=c_{k_{1}+1}$,  the image curve $\delta= f(\gamma)$  is a closed curve.  We can choose $\gamma$ once and for all  such that $\delta$ separates $0$ and $\lambda$;  that is, so that $\delta$ is a non-trivial curve closed curve in $\hat{\C} \setminus \{0, \lambda\}$.    

\subsection{Homotopy class}
\label{class}

The fundamental group $\pi_{1}(\hat{\C} \setminus \{0, \lambda\})={\mathbb Z}$ so  the homotopy class $\eta=[\delta]$  
in the fundamental group is an integer which essentially counts the number of fundamental domains 
between $c_{k_{1}}$ and $c_{k_{2}}$  and defines a ``distance'' between the fundamental domains. The integer $\eta$ depends only on the choice of $\gamma$ and since $\gamma$ is fixed, so is $\eta$.   

We now show that $\eta$ is an invariant of the Thurston iteration procedure and is thus a topological constraint on the iterates.  
 
\medskip 
\begin{lemma}\label{winding2}
Given $\tau_{0}=[\mu_{0}]\in T_f$, let $\tau_{n}=\sigma^{n} (\tau_{0}) =[\mu_{n}]$
be the sequence generated by $\sigma$. Let $w^{\mu_{n}}$ be the normalized quasiconformal map with Beltrami coefficient $\mu_{n}$ Let  $\gamma_{n+1}=w^{\mu_{n+1}}(\gamma)$, $\delta_{n}= w^{\mu_{n}}(\delta)$ and $\lambda_n=w^{\mu_{n}}(\lambda)$.    Then $[\delta_{n}] \in \pi_{1}(\hat{\C} \setminus \{0, \lambda_{n}\})= \eta$ for all $n$.   
\end{lemma}

\begin{proof}
The iteration defines the map 
$$
g_{n} = w^{\mu_{n}}\circ f\circ (w^{\mu_{n+1}})^{-1}\in \T2
$$
which is holomorphic since it  preserves $\mu_0$.  The continuous curve
$$
\gamma_{n+1}=w^{\mu_{n+1}} (\gamma)
$$
goes from  $c_{k_{1}, n+1}= w^{\mu_{n+1}} (c_{k_{1}})$ to $c_{k_{2}, n+1}= w^{\mu_{n+1}} (c_{k_{2}})$.  
The image curve
$$
\delta_{n}=g_{n} (\gamma_{n+1}) =   w^{\mu_{n}}(f((w^{\mu_{n+1}})^{-1} (\gamma_{n+1})))=w^{\mu_{n}} (f (\gamma)) = w^{\mu_{n}} (\delta)
$$
is a closed curve through the point $c_{k_{1}+1, n}=w^{\mu_{n}}(c_{k_{1}+1})$.

From our normalization, it follows that 
\begin{equation}\label{g}
g_{n}(z) =g_{\alpha_{n}, \beta_{n}}(z)=\frac{\alpha_{n} e^{\beta_{n}z}}{(\alpha_{n}-\frac{1}{\alpha_{n}})e^{\beta_{n}z} +\frac{1}{\alpha_{n}}}.
\end{equation}
and $0$ is an omitted value for $g_{n}$.   
Since $\lambda_n=w^{\mu_{n}}(\lambda)$,  it is 
also omitted for $g_{n}$ and
\begin{equation}\label{omit}
\lambda_{n} =\frac{\alpha_{n}}{\alpha_{n}-\frac{1}{\alpha_{n}}}\in P_{f,n} =w^{\mu_{n}} (P_{f}).
\end{equation}
Because
$$
w^{\mu_{n}}: \hat{\C}\setminus \{0, \lambda\}\to \hat{\C}\setminus \{0, \lambda_{n}\}
$$
is a normalized homeomorphism, it preserves homotopy classes and $\eta=[\delta_{n}] \in \pi_{1}(\hat{\C} \setminus \{0, \lambda_{n}\})={\mathbb Z}$. Thus the homotopy class of $\delta_{n}$ in the space $\hat{\C} \setminus \{0, \lambda_{n}\}$ is the same throughout the iteration. 
\end{proof}

Note that when $f\in \T2$ with $\lambda=\infty$, it is also in $\TE_{0,1}$. So the homotopy class defined in this section is the same as the winding number defined in subsection~\ref{winding}.

\section{Compactness.}
\label{compactness}
The arguments  that the invariance  under the Thurston iteration scheme of the winding number and the homotopy class together with the  bounded geometry condition  imply  compactness are different in the proofs of Theorem~\ref{main2} and Theorem~\ref{main3}.  We present these arguments in the two subsections below.  Recall that 
$$
P_{f, n} =w^{\mu_{n}} (P_{f}), \quad n=0, 1, 2, \ldots.
$$

\subsection{The proof of Theorem~\ref{main2}.}
\label{comp1} 
For such a map,  $f(0)=0$, $0$ is a branch point of multiplicity $p-1$ and $f$ has exactly one non-zero branch point $c$ with $f(c)=1$. All the functions in the Thurston iteration have the form
$$
E_{n} (z) = \alpha_{n} z^{p} e^{\lambda_{n} z}, \quad
\alpha_{n}=e^{p} \Big( -\frac{\lambda_{n}}{p}\Big)^{p}.
$$
Note that $E_{n}(0)=0$ and $0$ is a critical point of multiplicity $p-1$. It is also the asymptotic value and hence it has no other pre-images.  
Moreover, $E_{n}(z)$ has exactly one non-zero simple critical point 
$$
c_{n}=-\frac{p}{\lambda_{n}}=w^{\mu_{n}}(c)
$$
and $\alpha_{n}$ is defined by the normalization condition $E_{n}(c_{n})=1$.
    
If $c$ is periodic, then $c\in P_{f}$. This implies that $c_{n} (\not= 0, \infty) \in P_{f, n}$ and thus its spherical distance from either $0$ or $\infty$ is bounded below.  That is,  there are two constants $0<\kappa<K<\infty$ such that
$$
\kappa\leq |\lambda_{n}|\leq K, \quad \;\forall n>0.
$$   
This implies that the sequence $\{E_{n}\}_{n=1}^{\infty}$ is contained in a compact subset. 

Now suppose $c$ is not periodic. By the the hypotheses in Theorem~\ref{main2}, $f(c)=1$ is also not periodic. 
This implies that $k_{1}\geq 1$. 

We have 
$$
 0,  \quad 1=E_{n}(c_{n}), \quad E_{n}(1) = e^{p} \Big( -\frac{\lambda_{n}}{p}\Big)^{p} e^{\lambda_{n}} \in P_{f, n}.
$$
Let $c_{k, n} =w^{\mu_{n}}(c_{k})$. Then $c_{k,n}\in P_{f, n}$ for all $k\geq 1$.  
Let $\gamma_{n} =w^{\mu_{n}}(\gamma)$ and  $\delta_{n} =w^{\mu_{n}}(\delta)$ 

When $f$ has bounded geometry,  since $E_{n}(1) \not=0$, its spherical   distance from $0$ is bounded below.   This implies that  the sequence $\{|\lambda_{n}|\}$ is bounded below; that is, there is a constant $\kappa>0$ such that
$$
\kappa\leq |\lambda_{n}|, \;\; \forall n>0.
$$

By our hypothesis,  $c_{k_{1},n+1}\not= c_{k_{2},n+1}$ both belong to $P_{f, n+1}$ and bounded geometry  implies there are two constants, which we again  denote by $\kappa, K$ with 
$0<\kappa< K<\infty$, such  that 
$$
\kappa\leq |c_{k_{2}, n+1}|, |c_{k_{1}, n+1}|\leq K \;\;\hbox{and} \;\;  \kappa \leq  |c_{k_{2}, n+1}-c_{k_{1},n+1}| \leq K, \quad \forall\; n\geq 1.
$$
 
Now we prove that the sequence $\{|\lambda_{n}|\}$ is also bounded above. 
Recall that when we chose $\gamma$, we assumed it did not go through $0$ and thus by the normalization, 
none of the $\gamma_{n+1}$ go through $0$ either.  Therefore, for each $n$ we can find a simply connected domain 
$D_{n+1} \supset \gamma_{n+1}$ that does not contain $0$. 
Now we compute 
$$
\eta=  \frac{1}{2\pi i} \oint_{\delta_{n}} \frac{1}{w} dw = \frac{1}{2\pi i} \int_{\gamma_{n+1}} \frac{E_{n}'(z)}{E_{n}(z)} dz
= \frac{1}{2\pi i} \int_{\gamma_{n+1}} \Big(\frac{p}{z} +\lambda_{n}\Big)dz
$$
so that 
 $$
2 \pi  i \eta  = \int_{\gamma_{n+1}} \frac{p}{z} dz +\int_{\gamma_{n+1}} \lambda_{n} dz = \int_{\gamma_{n+1}} \frac{p}{z} dz + \lambda_{n} (c_{k_{2},n+1}-c_{k_{1},n+1}).
$$
Rewriting we have 
$$
\lambda_{n} (c_{k_{2},n+1}-c_{k_{1},n+1}) =2 \pi  i \eta - \int_{\gamma_{n+1}} \frac{p}{z} dz.
$$
This implies that
$$
\kappa |\lambda_{n}| \leq |\lambda_{n} (c_{k_{2},n+1}-c_{k_{1},n+1})| \leq  2 \pi  \eta + \Big|\int_{\gamma_{n+1}}\frac{p}{z} dz\Big|.  
$$
Thus,  if we can bound the integral on the right we will be done. 

Notice that  $\log z$ can be defined as an analytic function on the simply connected domain $D_{n+1}$ containing $\gamma_{n+1}$ that we chose above.  
We take $\log z =\log |z| +2\pi i \arg (z)$ as the principal branch, with $0\leq \arg (z)<2\pi.$  
We then estimate 
$$
\Big|\int_{\gamma_{n+1}}\frac{p}{z} dz\Big|=| \log c_{k_{2},n+1} - \log c_{k_{1},n+1}|
$$
$$
\leq |\log |c_{k_{2},n+1}|-\log |c_{k_{1},n+1}| | + |\arg(c_{k_{2},n+1}) - \arg (c_{k_{1},n+1})| 
$$
$$
\leq (\log K -\log \kappa) +4\pi .
$$
 Finally we have  
 $$
 |\lambda_{n}|  \leq \frac{2\pi \eta +(\log K-\log \kappa)  +4\pi}{\kappa}.
 $$ 
which proves that $\{ E_{n}(z)\}_{n=0}^{\infty}$ is contained in a compact subset in $\E_{p,1}$. 
This combined with Theorem~\ref{main1} completes the proof of Theorem~\ref{main2} for $f\in \TE_{p,1}$.  

\subsection{Proof of Theorem~\ref{main3}.}
\label{comp2}

By hypothesis $f$ has bounded geometry and by the normalization of $f$,   $\Omega_f=\{0, \lambda\}$, $f(0)=1$ so that $\{0,1, \lambda, \infty\} \subset P_f$.
Moreover the iterates $$
g_{n}=w^{\mu_{n}}\circ f\circ (w^{\mu_{n+1}})^{-1}
$$
belong to  $\M_2$.

Recall  that $P_{f, n} =w^{\mu_{n}} (P_{f})$ and because 
   $w^{\mu_n}$ fixes $\{0, 1, \infty\}$ for all $n\geq 0$,
$\{0, 1, \infty\}\subset P_{f, n}$.
By equation (\ref{g}),
$$
g_{n}(1) =w^{\mu_{n}}(f(1))=  \frac{\alpha_{n} e^{\beta_{n}}}{(\alpha_{n}-\frac{1}{\alpha_{n}})e^{\beta_{n}} +\frac{1}{\alpha_{n}}}\in P_{f, n}.
$$
so that
$$
\{0, 1, \lambda_{n}, g_{n}(1), \infty\} \subseteq P_{f, n}.
$$

In  the case that $\lambda=\infty$, $f$ is in ${\mathcal TE}_{0,1}$. Then we have that all the functions in the Thurston iteration have the form $g_{n} (z) = e^{\lambda_{n} z}$.   From our normalization, we have  
$$
0, \; 1=g_{n}(0),\;  g_{n}(1) =e^{\lambda_{n}} \in P_{f, n+1}.
$$  
If $\#(P_f)=3$, then $f(1)=1$. In this case, $g_{n}(1)=1$ for all $n\geq 0$  and $\#(P_{f, n})=3$ so that 
  $g_{n} (z)=e^{2\pi m_{n} z}$. 
The homotopy class of $\delta_{n}$ is determined by  $\eta$, the  winding number about the origin in the complex analytic sense. Thus $m_{n}=\eta$ for all $n$
and $g_{n}=e^{2\pi i \eta z}$, 
which is  fixed under  Thurston iteration and trivially lies in a compact  subset in $\E_{0,1}\subset \M2$.

Now suppose $\#(P_{f})\geq 4$. In this case $g_{n}(1) \neq1$.  When $f$ has bounded geometry,  
the spherical distance between $1$ and  $g_{n}(1)$ is bounded away from zero.  
That is,  there is a constant $\kappa>0$ such that
$$
\kappa \leq |\lambda_{n}|, \;\; \forall n\geq 0.
$$
 
Now we prove that the sequence $\{|\lambda_{n}|\}$ is also bounded above.  Again we compute
$$
\eta = \frac{1}{2\pi i} \oint_{\delta_{n}} \frac{1}{w} dw = \frac{1}{2\pi i} \int_{\gamma_{n+1}} \frac{g_{n}'(z)}{g_{n}(z)} dz= \frac{1}{2\pi i} \int_{\gamma_{n+1}} \lambda_{n}dz.
$$
The integral therefore depends only on the endpoints and  we have 
$$
\eta= \frac{1}{2\pi i} \int_{\gamma_{n+1}} \lambda_{n}dz=\frac{\lambda_{n}}{2\pi i} (c_{k_{2}, n+1}-c_{k_{1}, n+1}). 
$$  
Since $0$ is omitted, it cannot be periodic.  Therefore, both $c_{k_{2}, n+1}\not= c_{k_{1}, n+1}\in P_{f, n+1}$;  
 by bounded geometry therefore,   there is a positive constant which we again  denote by $\kappa$ such that 
$$
|c_{k_{2}, n+1}-c_{k_{1}, n+1}|\geq \kappa. 
$$ 
This gives us the estimate 
$$
|\lambda_{n}| \leq \frac{2\pi \eta}{|c_{k_{2}, n+1}-c_{k_{1}, n+1}|}\leq \frac{2\pi \eta}{\kappa} 
$$
which proves that $\{ g_{n}(z)\}_{n=0}^{\infty}$ is contained in a compact family in $\E_{0,1}\subset \M2$. 

Now let us prove  compactness of the iterates when $\lambda\not=\infty$. 
In this case, since
$$
\lambda_{n} =\frac{\alpha_{n}}{\alpha_{n}-\frac{1}{\alpha_{n}}}\in P_{f, n+1}
$$
has a definite spherical distance from $0$, $1$, and $\infty$, bounded geometry implies there are  two constants $0< k<K < \infty$ such that
$$ 
k\leq |\alpha_{n}|, \;\;|\alpha_{n} -1| \leq K, \quad \forall \; n\geq 0.
$$
In this case, we also know  that $g_{n} (1)\not=1$. Since $g_{n} (1) \in P_{f, n+1}$, bounded geometry implies that the constant $k$ can be chosen  such that
$$
k\leq |\beta_{n}|, \;\; \forall n>0.
$$

Again we use the topological constraint to prove that $\{|\beta_{n}|\}$ is also bounded from above. 
Let 
$$
M_{n} (z) =\frac{\alpha_{n} z}{(\alpha_{n}-\frac{1}{\alpha_{n}})z +\frac{1}{\alpha_{n}}}
$$
so that $g_{n}(z) = M_{n} (e^{\beta_{n}z})$.
The map $M_{n}: \hat{\C}\setminus \{ 0, \infty\} \to \hat{\C}\setminus \{0, \lambda_{n}\}$ is a homeomorphism so it induces an isomorphism from the fundamental group $\pi_{1}(\hat{\C}\setminus \{ 0, \infty\})$ to the fundamental group  $\pi_{1}(\hat{\C}\setminus \{ 0, \lambda_{n}\})$. Thus,  the homotopy class $[\tilde{\delta}_{n}]$ where $\tilde{\delta}_{n}= M_{n}^{-1} (\delta_{n})$, is given by an integer $\eta$.

 Note that $\tilde{\delta}_{n}$ is the image of $\gamma_{n+1}$ under $\widetilde{g}_{n}(z)=e^{\beta_{n} z}$. Since 
$\widetilde{\delta}_{n}$ is a closed curve in $\hat{\C}\setminus \{ 0, \infty\}$, $\eta$ is the winding number of $\tilde{\delta}_{n}$ about the origin  in the complex analytic sense,  and we can compute
$$
\eta = \frac{1}{2\pi i} \oint_{\widetilde{\delta_{n}}} \frac{dw}{w} =\frac{1}{2\pi i} \int_{\gamma_{n+1}} \frac{\tilde{g}_{n}'(z)}{\widetilde{g}_{n} (z)} dz=\frac{\beta_{n}}{2\pi i} (c_{k_{2}, n+1}-c_{k_{1}, n+1}).
$$  
As above,  $c_{k_{2}, n+1}, c_{k_{1}, n+1}\in P_{f, n+1}$,  and by bounded geometry there is a constant $k>0$ such that 
$$
|c_{k_{2}, n+1}-c_{k_{1},n+1}|\geq k,
$$ 
 so that 
$$
 |\beta_{n}| \leq \frac{2\pi \eta}{|c_{k_{2}, n+1}-c_{k_{1}, n+1}|}\leq \frac{2\pi \eta}{k}. 
$$
This inequality proves that $\{ g_{n} (z) = g_{\alpha_{n}, \beta_{n}}(z)\}$ forms a compact subset in $\M2$.

Finally, we  have shown that in all cases  the sequence $\{g_{n}\}$ is a compact subset in $\M2$. This combined with Theorem~\ref{main1} completes the proof of Theorem~\ref{main3}.

\vspace*{20pt}
\noindent Tao Chen, Department of Mathematics, Engineering and Computer Science, 
Laguardia Community College, CUNY,  
31-10 Thomson Ave. Long Island City, NY 11101.
Email: tchen@lagcc.cuny.edu

\vspace*{5pt}
\noindent Yunping Jiang, Department of Mathematics, Queens College of CUNY,
Flushing, NY 11367 and Department of Mathematics, CUNY Graduate
School, New York, NY 10016. 
Email: yunping.jiang@qc.cuny.edu

\vspace*{5pt}
\noindent Linda Keen, Department of Mathematics, Lehman College of CUNY,
Bronx, NY 10468
and Department of Mathematics, CUNY Graduate
School, New York, NY 10016.
Email: LINDA.KEEN@lehman.cuny.edu

\end{document}